\documentclass[12pt,twoside,a4paper]{amsart}
\usepackage[utf8]{inputenc}

\usepackage[inner=2.5cm,outer=3cm,top=2.5cm,bottom=4cm,includeheadfoot]{geometry}

\newcommand{\folder}{../../HeaderLiteratur/}

\usepackage{amsmath}
\usepackage{amssymb}
\usepackage{amsopn}
\usepackage{amsthm}
\usepackage{amsfonts}
\usepackage[right]{eurosym}
\usepackage{mathtools}

\usepackage{color} 
\usepackage{graphicx}

\usepackage{tikz}
\usepackage[all]{xy}

\usepackage{longtable}
\usepackage{totpages}

\usepackage[utf8]{inputenc} 
\usepackage[T1]{fontenc}

\usepackage{mathrsfs}
\usepackage{paralist}

\usepackage{longtable}

\definecolor{darkblue}{rgb}{0,0,0.6}
\usepackage[ocgcolorlinks,colorlinks=true, citecolor=darkblue, filecolor=darkblue, linkcolor=darkblue, urlcolor=darkblue]{hyperref}

\usepackage{tikz}
\usetikzlibrary{matrix,arrows,calc,snakes,patterns,decorations.markings}
\usepackage{tikz-cd}

\usepackage{tkz-euclide}
\tikzset{
  dotted/.style={pattern=dots,pattern color=#1},
  dotted/.default=black
}
\tikzset{
  fdotted/.style={pattern=crosshatch dots,pattern color=#1},
  fdotted/.default=black
}
\tikzset{
  scopedlines/.style={pattern=north east lines,pattern color=#1},
  scopedlines/.default=black
}
\tikzset{
  hrlines/.style={pattern=horizontal lines,pattern color=#1},
  hrlines/.default=black
}

\usepackage{eucal}
\usepackage{wasysym}

\usepackage{xspace} % Macht den Befehl \xspace verfügbar, der einen Blank einfügt, falls er nicht vor
\newcommand{\bN}{{\mathbb N}}

\newcommand{\bP}{{\mathbb P}}
\newcommand{\bQ}{{\mathbb Q}}
\newcommand{\bR}{{\mathbb R}}

\newcommand{\bZ}{{\mathbb Z}}

\newcommand{\cB}{{\mathscr B}}
\newcommand{\cC}{{\mathscr C}}

\newcommand{\cF}{{\mathscr F}}

\newcommand{\cH}{{\mathscr H}}

\newcommand{\cK}{{\mathscr K}}

\newcommand{\cM}{{\mathscr M}}

\newcommand{\cX}{{\mathscr X}}

\newcommand{\dO}{{\mathcal O}}

\newcommand{\dW}{{\mathcal W}}

\renewcommand{\phi}{\varphi}

\DeclareMathOperator{\iso}{\cong}
\newcommand{\twolines}[2]{\genfrac {}{}{0pt}{}{#1}{#2}}

\DeclareMathOperator{\surj}{\twoheadrightarrow}

\DeclareMathOperator{\too}{\longrightarrow}
\DeclareMathOperator{\rank}{rk}

\DeclareMathOperator{\Pic}{Pic} 
\DeclareMathOperator{\Div}{Div} 

\DeclareMathOperator{\NS}{NS}
\DeclareMathOperator{\del}{\partial}
\DeclareMathOperator{\birKbar}{\overline{\cB\cK}}

\DeclareMathOperator{\MonHdg}{Mon^2_\text{Hdg}}

\renewcommand{\div}{{\rm div}}

\DeclareMathOperator{\Nef}{Nef}
\DeclareMathOperator{\Amp}{Amp}

\DeclareMathOperator{\half}{\frac{1}{2}}
\DeclareMathOperator{\LambdaKEn}{\Lambda_{{\rm K3}^{[n]}}}

\DeclareMathOperator{\Movbar}{\overline{Mov}}

\newcommand\restr[2]{{% we make the whole thing an ordinary symbol
  \left.\kern-\nulldelimiterspace % automatically resize the bar with \right
  #1 % the function
  \vphantom{\big|} % pretend it's a little taller at normal size
  \right|_{#2} % this is the delimiter
  }}

\newif\ifmyversion
 \myversionfalse
\newcommand{\TODO}[1]{}%{\textcolor{red}{(Todo: {\it #1})}\ }

\newcommand{\Martin}[1]{}%{\textcolor{green}{(@Martin: {\it #1})}\ }

\newdir^{ (}{{}*!/-5pt/@^{(}}  %Für injektiv-Pfeile: kann man jetzt tippen @{^{ (}->} und dann ist das spacing schöner.
\newdir_{ (}{{}*!/-5pt/@_{(}}

\theoremstyle{plain}
\newtheorem{proposition}{Proposition}[section]

\newtheorem{lemma}[proposition]{Lemma}
\newtheorem{claim}[proposition]{Claim}

\newtheorem{corollary}[proposition]{Corollary}

\newtheorem{conjecture}[proposition]{Conjecture}

\newtheorem{theorem}[proposition]{Theorem}
\newtheorem*{theorem*}{Theorem}
\newtheorem*{conjecture*}{Conjecture}
\newtheorem*{proposition*}{Proposition}
\newtheorem*{corollary*}{Corollary}

\theoremstyle{definition}
\newtheorem{definition}[proposition]{Definition}

\newtheorem*{notation*}{Notation}

\newtheorem{remark}[proposition]{Remark}

\theoremstyle{remark}

\newenvironment{proofofclaim}{\begin{proof}[Proof of claim]}{ \end{proof}}

\newtheoremstyle{name}
   {}{}{\itshape}{}{\bfseries }{}{ }{\thmname{#3}.}
\theoremstyle{name}

\numberwithin{equation}{section}					% Nummerierung der Equations

\usepackage{fancyhdr}

\usepackage[final]{pdfpages}

\begin{document}

\title[Base divisors of big and nef line bundles on ISV]{Base divisors of big and nef line bundles on
  irreducible symplectic varieties}
\author[U. Rie\ss]{Ulrike Rie\ss}
\address{ETH Z\"urich, Institute of theoretical studies, Clausisusstrasse 47, 8092 Z\"urich, Switzerland}
\email{ulrike.riess@eth-its.ethz.ch}

\maketitle

\begin{abstract}
  Under some conditions on the deformation type, which we expect to be satisfied for arbitrary irreducible
  symplectic varieties, we describe which big and nef line bundles on
  irreducible symplectic varieties have base divisors. In particular, we show that
  such base divisors are always irreducible and reduced.
  This is applied to understand the behaviour of divisorial base components of big and nef line bundles under
  deformations and for  K3$^{[n]}$-type and Kum$^n$-type.
\end{abstract}

\section*{Introduction}
Irreducible (holomorphic) symplectic varieties are a class of
varieties that appears naturally in the classification of algebraic varieties with trivial first Chern
class. The Beauville--Bogomolov decomposition theorem (\cite[Theorem 1]{Beauville1983}) states that up to a finite \'etale cover
every such variety can be decomposed into a product of three types of varieties: abelian varieties, irreducible
symplectic varieties and (strict) Calabi--Yau varieties. Since this is known, there has been intense research
on irreducible symplectic varieties.

Two-dimensional irreducible symplectic varieties are exactly the famous K3 surfaces. These surfaces have a
 rich geometry and are very well-studied. It turns out, that many results on K3 surfaces stem from
 more general phenomena for irreducible symplectic varieties.

In this article
we study base divisors of big and nef line bundles on irreducible symplectic varieties.

Starting point are the known results on K3 surfaces:
\begin{proposition*}[{\cite{Mayer72}}]
  Let $X$ be a K3 surface with a big and nef line bundle $H \in \Pic(X)$.
Then $H$ has base points if and only if $H=mE+C$, where $m\geq2$, $E$ is a smooth elliptic curve, and
$C$ is a smooth rational curve, such that $(E, C)=1$. In this case, the base locus of $H$ is exactly $C$.
\end{proposition*}

  Although this article only deals with the complex setting, let us mention that over arbitrary
  algebraically closed fields of characteristic $\neq 2$ a similar 
  result for ample line bundles was proved by Saint-Donat (see \cite[Proposition 8.1]{Saint-Donat74}).

Mayer's proposition implies that for every ample line bundle $H$ on a K3 surface, $2H$ is base
point free.
Remarkably, also for abelian varieties $A$, it is known that  $2H$ is base point free for every ample line
bundle $H\in \Pic(A)$ (see \cite{MumfordAV}).

These results can be seen as part of a general conjecture of Fujita which predicts that for smooth, projective
varieties  $\omega_X+(\dim X + 1) H $ is base point free for any ample line bundle $H$ (see
e.g.~\cite[Conjecture 10.4.1]{LazarsfeldII} for the full statement of Fujita's conjecture).
Note that for abelian varieties and for K3 surfaces, the bounds are even better than predicted by
Fujita's conjecture. This suggests, that it might be particularly interesting to study questions related
to base points of ample line bundles for irreducible symplectic varieties.

Let us mention that in general it is known, that for an ample line bundle $H$ on a smooth projective
variety $X$ of dimension 
$n$, the line bundle $\omega_X+ m H$
is globally generated for $m\geq \binom{n+1}{2}$ (see \cite{Angehrn-Siu}). 

\bigskip

In this article, we investigate the divisorial part of the base locus of big and nef line bundles on
 irreducible symplectic varieties. The main theorem is

\begin{theorem*}[{compare Theorem \ref{thm:main-thm}}]
    Fix an irreducible symplectic variety $X$ which satisfies certain conditions
    (in particular this holds for all $X$ of K$3^{[n]}$--type or of Kum$^n$--type).
  Consider a big and nef line bundle $H \in \Pic(X)$.
  Then $H$ has a non-trivial base divisor if and only if there exists an irreducible reduced divisor $F$
  of negative square such that $H$ is of the form $H=mL+F$, where $m\geq 2$, $L$ is a
  primitive movable line bundle with $q(L)=0$ and  $(L,F)_q>0$, such that $RR_X(q(H))=\binom{m+n}{n}$.
  In this case $F$ is exactly the fixed divisor of $H$.
\end{theorem*}

We expect that the conditions, under which the theorem holds, are satisfied for arbitrary irreducible
symplectic varieties. However, they are very hard to check and include a famous conjecture.

Note that in particular the base divisor is irreducible and reduced of negative square, whenever the theorem
applies, which is  
a remarkable observation in itself.

One obtains the following corollary which fits in the framework of Fujita's conjecture.
\begin{corollary*}[see Corollary \ref{cor:2HnobasecompforISV}]
  Fix an irreducible symplectic variety $X$ which satisfies the assumptions of the theorem.
  Pick a line bundle $H\in \Pic(X)$ which is big and nef. Then
  the base locus of $2H$ does not contain a divisor.
\end{corollary*}

The main theorem can be used for K3$^{[n]}$-type and Kum$^n$-type to describe big and nef line bundles with
non-trivial base locus even more explicitly (see Propsition \ref{prop:basecomponentsK3n-type} and \ref{prop:basecomponentsKumn-type}). As further application of the main
theorem one understands the behaviour of divisorial base loci under 
deformations: 
\begin{theorem*} [{see Theorem \ref{thm:defoofBD}}]
  Let $X$ be an irreducible symplectic variety such that its deformation type satisfies the conditions in
  the main theorem. Then in a (semi-)polarized family the locus where the big and nef line bundle acquires a
  divisorial base locus is a disjoint union of certain Noether--Lefschetz divisors.
\end{theorem*}

In this article, we only discuss the phenomenon of divisorial base loci.
For K3 surfaces, Mayer's proposition shows that only divisorial base loci can occur for big and nef line
bundles. However, for higher 
dimensional irreducible symplectic varieties, it turns out that the situation is more complicated. Therefore,
we restrict ourselves to divisorial base loci in this article and present first results towards base loci of
higher codimension in the next article, which will follow shortly.

\bigskip

In Sections \ref{sec:basicfacts} to \ref{sec:reflectionsinPED}, we state  properties of irreducible symplectic varieties, including results on
the birational Kähler cone and reflections in prime exceptional divisors.

The main theorem of this article (Theorem \ref{thm:main-thm}) is stated in its most general form in Section \ref{sec:main} and
proved in Section \ref{sec:proofofmainthm}. A discussion of the non-standard assumptions in the theorem can be found in Section
\ref{sec:discuss-conditions}. 

Section \ref{sec:defo} contains the application of the main theorem to the behaviour under deformations. The
explicit description of base divisors for K3$^{[n]}$-type and Kum$^n$-type can be found in Sections \ref{sec:basedivforK3n}
and \ref{sec:basedivforKumn}.

\section*{Acknowledgements}
Most of this article was part of my dissertation. I thank my advisor Daniel Huybrechts for his great
supervision and the SFB TR 45 at University of Bonn for the financial support during my PhD.
I am also grateful for the financial 
support by ERC and the Institute of Theoretical Studies at ETH Zürich during the completion and the editing of
the article. I am thankful to Claire Voisin for the exchange and encouragement during this final period.
Further, I would like to thank my husband for his wonderful support.

%%%%%%%%%%%%%%%%%%%%%%%%%%%%%%%%%%%%%%%%%%%%%%%%%%%%%%%%%%%%%%%%%%%%%%%%%%%%%%%%%%%%%%%%%%%%%%%%%%%%
%%%%%%%%%%%%%%%%%%%%%%%%%%%%%%%%%%%%%%%%%%%%%%%%%%%%%%%%%%%%%%%%%%%%%%%%%%%%%%%%%%%%%%%%%%%%%%%%%%%%
%%%%%%%%%%%%%%%%%%%%%%%%%%%%%%%%%%%%%%%%%%%%%%%%%%%%%%%%%%%%%%%%%%%%%%%%%%%%%%%%%%%%%%%%%%%%%%%%%%%%
%%%%%%%%%%%%%%%%%%%%%%%%%%%%%%%%%%%%%%%%%%%%%%%%%%%%%%%%%%%%%%%%%%%%%%%%%%%%%%%%%%%%%%%%%%%%%%%%%%%%

%%%%%%%%%%%%%%%%%%%%%%%%%%%%%%%%%%%%%%%%%%%%%%%%%%%%%%%%%%%%%%%%%%%%%%%%%%%%%%%%%%%%%%%%%%%%%%%%%%%%
%%%%%%%%%%%%%%%%%%%%%%%%%%%%%%%%%%%%%%%%%%%%%%%%%%%%%%%%%%%%%%%%%%%%%%%%%%%%%%%%%%%%%%%%%%%%%%%%%%%%

\section{Basic facts on irreducible symplectic varieties}\label{sec:basicfacts}
This section contains a collection of basic facts on irreducible symplectic varieties,
which we need later.

Let us start by giving the definition:
\begin{definition}
  An {\it irreducible symplectic variety} is a simply connected, smooth, projective complex variety $X$ such
  that  $H^0(X,\Omega_X^2)$ is generated by a nowhere degenerate two-form. 
\end{definition}

Irreducible symplectic varieties are also known as ``projective hyperkähler manifolds'' and 
as ``irreducible holomorphic symplectic varieties''.
For an
overview on irreducible symplectic varieties, we refer to \cite[Part III]{Gross-Huybrechts-Joyce},
\cite{O'Grady:Intro}, and
\cite{Huybrechts:HK:basic-results}.

All known irreducible symplectic varieties are deformation equivalent to one of the following:
\begin{itemize}[--]
\item the $2n$-dimensional Hilbert scheme Hilb$^n(S)$ of $n$ points on a K3 surface $S$, 
\item the $2n$-dimensional generalized Kummer variety Kum$^n(A)$ associated to an abelian surface $A$,
  which is constructed as 
  a fibre of the summation map Hilb$^{n+1}(A)\to A$,
\item two examples discovered by O'Grady (a ten-dimensional one \cite{O'Grady:10-dim} and a six-dimensional one \cite{O'Grady:6-dim}).
\end{itemize}
Irreducible symplectic varieties that
are deformation equivalent to one of the first two series of examples are called {\it K3$^{[n]}$-type}
respectively {\it Kum$^n$-type}.

We will usually denote the even dimension of an irreducible symplectic variety $X$ by $2n$, and we will not
distinguish between $\Pic(X)$ and $\NS(X)$ which can be naturally identified.

%%%%%%%%%%%%%%%%%%%%%%%%%%%%%%%%%%%%%%%%%%%%%%%%%%%%%%%%%%%%%%%%%%%%%%%%%%%%%%%%%%%%%%%%%%%%5

\bigskip

We denote by $q$ the Beauville--Bogomolov--Fujiki quadratic form on 
the second integral cohomology of an irreducible symplectic variety $X$,
which satisfies the Fujiki relation
$\int_X \alpha^{2n} = C_X \cdot q(\alpha)^n$ for every $\alpha \in H^2(X,\bZ)$ with a constant $C_X\in \bQ_{>0}$ depending on $X$ (\cite{Fujiki87}).

The signature of $q$ is $(3, b_2 - 3)$, where $b_2=\rank H^2(X,\bZ)$. Restricted to $H^{1,1}(X,\bR)$ the
signature of $q$ is $(1,b_2-3)$.

Therefore one can define the {\it positive cone} $\cC_X \subseteq H^{1,1}(X,\bR)$ as the connected component of $\{\alpha \in H^{1,1}(X,\bR) \mid q(a)>0\}$
containing an ample class.
Note that the positive cone  contains all ample classes and that a nef class is big if and only if $q(H)>0$.

\bigskip
Another important property of $q$ is the following:

\begin{lemma}[{\cite[Proposition 4.2.(ii)]{Boucksom04}}]\label{lem:BBforeffective}
  Let $X$ be an irreducible symplectic variety, and $E,F \in \Pic(X)$ be effective divisors with no
  common component, then $q(E,F)\geq 0$.
\end{lemma}

Using the Beauville--Bogomolov--Fujiki form, the Hirzebruch--Riemann--Roch formula takes a special form for
irreducible symplectic varieties: 
\begin{theorem}[{Riemann--Roch for irreducible symplectic varieties, see \cite[Section 1.11]{Huybrechts:HK:basic-results}}]\label{thm:RRISV} 
For an irreducible symplectic variety $X$ there exist $b_i\in \bQ$  such that for each $L\in \Pic(X)$
\begin{equation*}
  \chi(X,L)=\sum_{i=0}^n b_i q(L)^i.
\end{equation*}
\end{theorem}

Let us further mention the following result:
\begin{proposition}[{\cite[Theorem 5.9]{Gross-Huybrechts-Joyce}}] \label{prop:chi0}
  Let $X$ be an irreducible symplectic variety of dimension $2n$. Then $\chi(\dO_X) = n+1$.
\end{proposition}

We will frequently use the following basic observation.
\begin{lemma}[{\cite[Section 2.2]{Huybrechts:birHK_deformations}}]\label{lem:isocodim1}
  Let $X$ and $X'$ be birational irreducible symplectic varieties. Then there exist open subsets $U\subseteq X$ and
  $U'\subseteq X'$ with
  $U\iso U'$, such that $X\setminus U$ and $X'\setminus U'$ have codimension at least two.
\end{lemma}

%%%%%%%%%%%%%%%%%%%%%%%%%%%%%%%%%%%%%%%%%%%%%%%%%%%%%%%%%%%%%%%%%%%%%%%%%%%%%%%%%%%%%%%%%%%%%%%%%
\section{The birational Kähler cone}\label{sec:cones}
In this section, we collect some useful facts about the birational Kähler cone of an irreducible symplectic variety.

\begin{definition}\label{def:birKbar}
  Let $X$ be an irreducible symplectic variety of dimension $2n$ as before.  Denote its Kähler cone by
  $\cK_X\subseteq H^{1,1}(X,\bR)$. Define its {\it birational Kähler cone} as
  \begin{equation*}
    \cB\cK_X\coloneqq  \bigcup_f f^* (\cK_{X'})\subseteq H^{1,1}(X,\bR),
  \end{equation*}
  where the union is taken over all birational maps $f\colon X\dashrightarrow X'$ from $X$ to another
  irreducible symplectic variety $X'$. Denote its closure by
  $\birKbar\hspace{-0.4em}_X\subseteq  H^{1,1}(X,\bR)$.
\end{definition}
 Note that the pullback along $f\colon X\dashrightarrow X'$ is well-defined, since the
  indeterminacy locus is of codimension at least two by Lemma \ref{lem:isocodim1}.

  \begin{theorem}[{\cite[Proposition 4.2]{Huybrechts2003}}]\label{thm:descrofbirKbar}
    Let $X$ be an irreducible symplectic variety. Then the closure of the birational Kähler cone of $X$
    can be described in the following way
    \begin{equation*}
      \birKbar_X
= \{\alpha \in \overline{\cC_X} \subseteq H^{1,1}(X,\bR) \mid (\alpha, D)_q \geq 0, \ \forall D\subseteq
X \text{ uniruled divisor}\}.
    \end{equation*}
  \end{theorem}

\begin{corollary}\label{cor:BBforbig+nef}
  Let $X$ be an irreducible symplectic variety, $E \in \Pic(X)$ an effective divisor and $H\in
  \birKbar_X$. Then $(H,E)_q\geq 0$. In particular this applies for all nef line bundles
  $H\in \Pic(X)$.
\end{corollary}
\begin{proof}
  This was actually part of the proof of \cite[Proposition 4.2]{Huybrechts2003}.
\end{proof}
 
\begin{remark}\label{rem:negsquareinThm}
  In  Theorem \ref{thm:descrofbirKbar}, we can restrict ourselves to considering irreducible effective
  divisors $D$ such that  
     $(D,\beta)_q<0$ for some $\beta \in  \overline{\cC_X}$.  
    Since $D$ satisfies  $(\alpha, D)_q >0$ for all $\alpha
    \in \birKbar_X\subseteq \overline{\cC_X}$ (by Corollary \ref{cor:BBforbig+nef}), this implies that
    $D$ is orthogonal to some element in $\cC_X$ in this case.
    Therefore the divisor $D$ satisfies $q(D)<0$. In this situation it was shown in \cite{Boucksom04},
    that $D$ is always uniruled:
  \end{remark}

  \begin{proposition}[{\cite[Proposition 4.7]{Boucksom04}}]\label{prop:boucksom}
    Every irreducible effective divisor $D\in \Pic(X)$ with $q(D)<0$ is uniruled.
  \end{proposition}

One can deduce the following standard fact:
  \begin{corollary}\label{cor:MovbirKbar}
    Fix an irreducible symplectic variety $X$.
    Let $\Movbar(X)$ be the closure of the movable cone in $\Pic(X)_\bR$, i.e.~the closure of the set 
    \begin{equation*}
      \{r\cdot M \in \Pic(X)_\bR\mid r\in \bR, \ M \in \Pic(X) \text{\ movable line bundle}\}.
    \end{equation*}
    Then $\Movbar(X)=\birKbar_X\cap \Pic(X)_\bR$.
  \end{corollary}
  \begin{proof}
    By Lemma \ref{lem:BBforeffective} and Theorem \ref{thm:descrofbirKbar} every movable line bundle
    $M\in \Pic(X)$ is contained in $\birKbar_X$.
    Conversely, every element  $\alpha \in \cB\cK_X\cap \Pic(X)$ is the pullback of an ample class on a
    birational model and therefore movable.
  \end{proof}

In general, the union in the definition of $\cB\cK_X$ can be infinite. Therefore it is important to observe
the following:
\begin{proposition}\label{prop:nefonbiratmodel}
  Let $X$ be an irreducible symplectic variety and fix $L\in \Pic(X)$. Suppose $L\in \Pic(X)\cap \birKbar_X\cap \cC_X$. Then
  there exists a birational irreducible symplectic variety $X'$  such that the associated line bundle $L'\in
  \Pic(X')$ is nef.
   If $b_2(X)\neq 4$, the same holds for $L\in \Pic(X)\cap \birKbar_X$ with $q(L)=0$.
\end{proposition}
  The proof uses the theory of wall divisors (also known as MBM-classes), which can be used to describe the
  interior structure of $\cB\cK_X$.

  Let $\dW'$ be the set of divisors $W \in \Pic(X)$ such that $W^\perp$ contains a face of one of the Kähler
  chambers in $\cB\cK_X$,
  i.e.~there is a birational map $f\colon X \dashrightarrow X'$ from $X$ to another irreducible symplectic
  variety $X'$ such that  $W^{\perp}\cap f^*(\cK_{X'})$ contains an open subset of $W^\perp$.
  \begin{definition}[{Wall-divisors/MBM-classes; compare \cite[Definition 1.2]{Mongardi13}, \cite[Definition 1.13]{AmerikVerbitsky14}}]
    The set $\dW$ of wall divisors on $X$ (also known as MBM-classes) is the union of $\phi(\dW')$ for all
     $\phi \in \MonHdg(X)$, where $\MonHdg(X)$ denotes the monodromy operators on $H^2(X,\bZ)$ which preserve
    the Hodge structure. 
  \end{definition}
  % Giovannis definition would admit limits of infinite wall divisors.
  % But these cannot occur by [AV, Theorem 6.2]

  An important property of wall divisors is
  \begin{proposition}[{\cite[Theorem 1.19]{AmerikVerbitsky14}}]\label{prop:walldivcutK}
    Let $X$ be an irreducible symplectic variety. The Kähler cone $\cK_X$ is a connected component of
    $\cC_X\setminus \bigcup_{W\in \dW} W^\perp$.
     Moreover, for every element $\alpha \in \cC_X\setminus \bigcup_{W\in \dW} W^\perp$ there exists
    $\phi\in \MonHdg(X)$, such that $\phi(\alpha)\in f^*(\cK_X')$ for a birational irreducible symplectic variety
    $X'$ as above.  
  \end{proposition}

  In order to deal with the case $q(L)=0$ of Proposition \ref{prop:nefonbiratmodel}, we further need the
  following:
  
\begin{theorem}\label{thm:boundedwalldiv}
  Let $X$ be an irreducible symplectic variety with $b_2(X)\neq 4$.
  Then the set $\{q(W) \mid W\in h^{1,1}(X) {\rm\ (primitive)\ wall\ divisor\ on \ }X\}$ is
  bounded.
\end{theorem}
\begin{proof}
  This is due to Amerik and Verbitsky.  Their theorem \cite[Theorem 3.17]{AmerikVerbitsky16X} states that for any
  irreducible symplectic variety $X$ with $b_2(X)\geq 5$ the monodromy
  group acts on the set of primitive wall divisors  with finitely many orbits. Indeed they show this even for the larger set of such classes
  in $H^2(X,\bZ)$. Since monodromy operators are isometries for $q$, this implies
  the statement.
\end{proof}

The other central ingredient is the following basic proposition:
\begin{proposition}[{\cite[Proposition 3.4]{MarkmanYoshioka14}}]\label{prop:finiteperp}
  Let $X$ be an irreducible symplectic variety, $\Pi \subseteq \overline{\cC_x}$ a (closed) rational polyhedral
  cone, and $N\in \bN$ a fixed natural number. Then the set
  \begin{equation*}
    \{w \in H^{1,1}(X,\bZ) \mid -N<q(w)<0 {\rm\ and\ }  w^{\perp}\cap \cC_X \cap \Pi \neq \emptyset \}
  \end{equation*} 
is finite.
\end{proposition}
Using these results, we can prove 
Proposition \ref{prop:nefonbiratmodel}:
\begin{proof}[{Proof of Proposition \ref{prop:nefonbiratmodel}}]
  From Proposition \ref{prop:walldivcutK}, one can deduce that
  \begin{equation*}
    \birKbar_X\cap \cC_X=\bigcup_f f^* (\overline{\cK_{X'}}\cap \cC_X),
  \end{equation*}
  where the union is again taken over
  all $f\colon X\to X'$, where $X'$ is a birational irreducible symplectic variety. This was not a priori clear, since the
  union could be infinite.
    Therefore, a line bundle in the intersection $L \in \Pic(X)\cap \birKbar_X\cap \cC_X$ lies in one of the $f^*(\overline{\cK_{X'}})$. Thus
    $L'\coloneqq {f^{-1}}^*(L) \in \overline{\cK_{X'}}$ is nef.

    It only remains to deal with the case $q(L)=0$ under the assumption that $b_2(X)\neq 4$.
  Fix $L\in \birKbar_X\cap \Pic(X)\cap \del\cC_X$. We need to show that the associated line bundle $L'$ is nef on some
  birational model.
  
  Pick any (closed) rational polyhedral cone $\Pi\subseteq \birKbar_X\cap \Pic(X)_\bR$ of dimension $\rho(X)\coloneqq h^{1,1}(X,\bZ)$ such that $L \in \Pi$ (e.g.~one
  can construct this by picking any rational polyhedral cone $\Pi'$ in $\Nef(X)=\overline{\Amp(X)}$ of
  dimension $\rho(X)$, and
  choosing $\Pi \coloneqq \Pi' + \bR_{\geq 0} L$).

  By  Theorem \ref{thm:boundedwalldiv}, the square of primitive wall divisors is bounded, and therefore we can
  apply Proposition \ref{prop:finiteperp} to deduce that 
    \begin{equation*}
    \{w \in H^{1,1}(X,\bZ) \mid w {\rm\ is\ wall\ divisor\ and\ }  w^{\perp}\cap \cC_X \cap \Pi \neq \emptyset
    \}
  \end{equation*} 
  is a finite set. 

  In particular $\Pi$ decomposes into a finite union of chambers of the form $\Pi \cap f^*(\Nef(X'))$,
  where $f\colon X \dashrightarrow X'$ are birational irreducible symplectic varieties.
  Therefore, $L$ lies in $f^*(\Nef(X'))$ for one such $X'$. Thus the associated line bundle
  $L' \coloneqq {f^{-1}}^*(L)$ lies in $\Nef(X')$ as claimed.
\end{proof}

\begin{remark}
  The same arguments show: If $b_2(X)\neq 4$ and $\birKbar_X$ is rationally polyhedral, then there are only finitely many
  chambers in $\birKbar_X$.
\end{remark}

The following useful property of line bundles in $A\in \Pic(X)\cap \birKbar_X$ with $q(A)\neq 0$ follows
from Proposition \ref{prop:nefonbiratmodel}:
\begin{lemma}\label{lem:h0inbirKbar}
  Let $X$ be an irreducible symplectic variety and pick a line bundle 
$A \in \Pic(X)\cap\birKbar_X\cap \cC_X$. Then 
  $h^0(X,A)=\chi(X,A)$. 
\end{lemma}
\begin{proof}
  Since $A \in \Pic(X)\cap \birKbar_X\cap\cC_X$, Proposition \ref{prop:nefonbiratmodel} implies that  there exists a birational irreducible symplectic variety $X'$ such that the corresponding
  line bundle $A'\in \Pic(X')$ is nef.  Then  $A'$ is indeed big and
  nef since $q(A')=q(A)>0$.
Consequently
\begin{align*}
  h^0(X,A)=h^0(X',A')=\chi(X',A')=\chi(X,A),
\end{align*}
where the first equality uses that $X$ and $X'$ are
isomorphic away from codimension two by Lemma \ref{lem:isocodim1}.
The second equality exploits that  $A'$ is big and nef, and therefore satisfies
Kodaira vanishing, and the last equality holds since $(X,M)$ and $(X',M')$ are deformation
equivalent (see \cite[Theorem 4.6]{Huybrechts:HK:basic-results}).
\end{proof}

%%%%%%%%%%%%%%%%%%%%%%%%%%%%%%%%%%%%%%%%%%%%%%%%%%%%%%%%%%%%%%%%%%%%%%%%%%%%%%%%%%%%%%%%%%%%%%%%%%%%%%%
%%%%%%%%%%%%%%%%%%%%%%%%%%%%%%%%%%%%%%%%%%%%%%%%%%%%%%%%%%%%%%%%%%%%%%%%%%%%%%%%%%%%%%%%%%%%%%%%%%%%%%%
\section{Reflections for  prime exceptional divisors}\label{sec:reflectionsinPED}
In this section we use Huybrechts' description of $\birKbar_X$ and Markman's results on reflections for
prime exceptional  divisors to prove Proposition
\ref{prop:RefltoBK}, which shows that arbitrary elements $\alpha \in \NS(X)_\bQ \cap \del \cC_X$ can be
moved into $\birKbar_X$, by a series of such reflections.

Let us first introduce some notations.

\begin{definition}
  \begin{enumerate}
  \item A divisor $D \in \Div(X)$ is called {\it prime exceptional divisor} if $D$ is reduced and
    irreducible and satisfies $q(D)<0$.
  \item For a prime exceptional divisor $D$ define the reflection $R_D\in O(H^2(X, \bQ))$ as
    \begin{equation*}
      R_D(\alpha)\coloneqq  \alpha - \frac{2(D,\alpha)_q}{q(D)}  \, D.
    \end{equation*}
  \end{enumerate}
\end{definition}

\begin{remark}
Several statements in the previous sections involved prime exceptional divisors. In particular Proposition
\ref{prop:boucksom}.
\end{remark}

For the proof of Proposition \ref{prop:RefltoBK}, we will need the following result of Markman:
\begin{proposition}[{\cite[Proposition 6.2]{Markman11}}]\label{prop:RisMonOP}
  For any prime exceptional divisor $D$ the reflection $R_D$ restricts to an integral morphism. In fact, it is
  a monodromy operator preserving the Hodge structure.
\end{proposition}

\begin{proposition}\label{prop:RefltoBK}
  Let $0\neq \alpha \in \NS(X)_\bQ\cap \overline{\cC_X}$  then there exists a composition $R=R_{D_{k-1}}\circ
  \dots \circ R_{D_0}$ of reflections
  $R_{D_i}$ associated to prime exceptional divisors $D_i$,
  such that $R(\alpha)\in \birKbar_X$.
\end{proposition}

\begin{proof} This proof is similar to the analogue for K$3$ surfaces, as presented in
  \cite[Remark VIII.2.12]{HuybrechtsK3}.

  By passing to a multiple of $\alpha$, we may assume that $\alpha \in H^2(X,\bZ)$ is an integral
  element.   Set $\alpha_0\coloneqq \alpha$.
  Fix an ample class $h\in H^2 (X,\bZ)$.
  For any element $\alpha_i\in H^2(X,\bZ)\cap \overline\cC_X$, the Beauville--Bogomolov pairing
  $(\alpha_i,h)_q$ is a positive integer. If $\alpha_i\notin \birKbar_X$, then by Theorem
  \ref{thm:descrofbirKbar} and Remark 
  \ref{rem:negsquareinThm} there exists a prime exceptional divisor $D_i$ with $(\alpha_i,D_i)_q <0$.
  Note that $(D_i,h)_q>0$ and $(\alpha_i,h)_q>0$, since $h$ is ample. 
  Set $\alpha_{i+1}\coloneqq R_{D_i}(\alpha_i)$, which is an element in the integral lattice
  $H^2(X,\bZ)$, since $R_{D_i}$ is an integral morphism by Proposition \ref{prop:RisMonOP}. Observe that
  \begin{align*}
    (\alpha_{i+1},h)_q
    = \big(R_{D_i}(\alpha_i),h \big)_q
    &= \big(\alpha_i- \frac{2(D_i,\alpha_i)_q}{q(D_i)}  D_i\, ,h \big)_q\\
    &=(\alpha_i,h)_q - \underbrace{\frac{2(D_i,\alpha_i)_q}{q(D_i)}\cdot (D_i,h)_q}_{>0}
    < (\alpha_i,h)_q \,.
  \end{align*}
  If $\alpha_{i+1}\notin \birKbar_X$, repeat the above for $\alpha_{i+1}$. 
  Since $(\alpha_0,h)> (\alpha_1,h)>(\alpha_2,h)> \dots$ is a descending sequence of positive integers,
  this procedure needs to stop for some $k\in \bN$, which implies that $\alpha_k\in \birKbar_X$. Set
  $R\coloneqq   R_{D_{k-1}} \circ \dots \circ R_{D_0}$. This concludes the proof, since
  $R(\alpha)=\alpha_k \in \birKbar_X$. 
\end{proof}

Furthermore, we will need the following inequality for exceptional prime divisors:

\begin{lemma}[{follows from \cite[Lemma 3.7]{MarkmanPED}}]\label{lem:ineqforPED}
  Let $X$ be an irreducible symplectic variety. Suppose $D\in \Pic(X)$ is an irreducible (and reduced)
  divisor with $q(D)<0$. Recall that $\div(D)\coloneqq \gcd \{(\alpha,D)_q \mid \alpha\in H^2(X,\bZ)\}$.
  Then $q(D)\,|\,2\div(D)$. In particular
  \begin{equation*}
      -\frac{1}{2} q(D) \leq \div(D).
  \end{equation*}
\end{lemma}

%%%%%%%%%%%%%%%%%%%%%%%%%%%%%%%%%%%%%%%%%%%%%%%%%%%%%%%%%%%%%%%%%%%%%%%%%%%%%%%%%%%%%%%%%%%%%%%%%

%%%%%%%%%%%%%%%%%%%%%%%%%%%%%%%%%%%%%%%%%%%%%%%%%%%%%%%%%%%%%%%%%%%%%%%%%%%%%%%%%%%%%%%%%%%%%%%%%%%%
%%%%%%%%%%%%%%%%%%%%%%%%%%%%%%%%%%%%%%%%%%%%%%%%%%%%%%%%%%%%%%%%%%%%%%%%%%%%%%%%%%%%%%%%%%%%%%%%%%%%
%%%%%%%%%%%%%%%%%%%%%%%%%%%%%%%%%%%%%%%%%%%%%%%%%%%%%%%%%%%%%%%%%%%%%%%%%%%%%%%%%%%%%%%%%%%%%%%%%%%%
%%%%%%%%%%%%%%%%%%%%%%%%%%%%%%%%%%%%%%%%%%%%%%%%%%%%%%%%%%%%%%%%%%%%%%%%%%%%%%%%%%%%%%%%%%%%%%%%%%%%

%%%%%%%%%%%%%%%%%%%%%%%%%%%%%%%%%%%%%%%%%%%%%%%%%%%%%%%%%%%%%%%%%%%%%%%%%%%%%%%%%%%%%%%%%%%%%%%%%%%%
%%%%%%%%%%%%%%%%%%%%%%%%%%%%%%%%%%%%%%%%%%%%%%%%%%%%%%%%%%%%%%%%%%%%%%%%%%%%%%%%%%%%%%%%%%%%%%%%%%%%

\section{Main Theorem}\label{sec:main}
In this section, we state the main theorem of this article and deduce some immediate consequences.

\bigskip

For the assumptions of the theorem in its most
general form, we need to introduce some notation first.

Let $X$ be an irreducible symplectic variety. Recall that by the Riemann--Roch theorem for irreducible
symplectic varieties (Theorem \ref{thm:RRISV})
 there exist $b_i\in \bQ$  such that for each $L\in \Pic(X)$
\begin{equation*}
  \chi(X,L)=\sum_{i=0}^n b_i q(L)^i.
\end{equation*}

\begin{definition}
  For a given irreducible symplectic variety $X$, denote the polynomial appearing in the Riemann--Roch formula
  by $RR_X(x)\coloneqq \sum_{i=0}^n b_i x^i$, where the $b_i$ are as above.
\end{definition}

 The polynomial $RR_X$ only depends on the deformation type of $X$.

\begin{remark}
  For the main theorem (Theorem \ref{thm:main-thm}), we will need to assume that $RR_X$ is strictly monotonic. Note that without any assumptions, it is not even clear that for an arbitrary irreducible symplectic variety $X$ and
  an ample line bundle $H\in \Pic(X)$ there are non-trivial global sections in $H^0(X,H)$. 
  Our assumptions on $RR_X$ will exclude this problem.
\end{remark}

Another statement which we will frequently need for the main theorem is that the following conjecture holds for the
variety $X$ which we consider:
\begin{definition}
  A line bundle $L\in \Pic(X)$ is said to {\it induce a rational Lagrangian fibration to $\bP^n$} if there exists a
  birational map $f\colon X \dashrightarrow X'$ to an irreducible symplectic variety $X'$, and 
  a fibration $\phi\colon X'\surj \bP^n$, such
  that the birational transform $L'$ of $L$ on $X'$ is the pullback of an ample line bundle.
\end{definition}
\begin{conjecture}\label{conj:strongRLF}
  Let $X$ be an irreducible symplectic variety and $0\neq L \in \Pic(X)\cap \birKbar_X$ be a primitive line bundle with
  $q(L)=0$. Then $\dim h^0(X,L)=n+1$ and $|L|$ induces a birational Lagrangian fibration to $\bP^n$.
\end{conjecture}

This is a classical conjecture on irreducible symplectic varieties. There is a  weaker version which does not make any predictions on the base space of the Lagrangian
fibration. For its history we refer to
\cite[p.~3]{VerbitskySYZ}.
Note however, that the base space of a Lagrangian fibration is automatically isomorphic to  
$\bP^n$ if it is smooth (see \cite{Hwang08} and \cite{GrebLehn14}). Smoothness of the base is not known in
general, but also conjectured.
Furthermore, the presented statement of the conjecture implies that the pullback $\phi^*\dO(1)$ is
primitive. This property is conjectured to hold in general and  has been subject of
\cite{KamenovaVerbitsky16}.
The version of the conjecture stated here also implies the existence of the birational model $X'$ on which
$L'$ is nef. By 
Proposition \ref{prop:nefonbiratmodel} this is automatic whenever $b_2(X)\neq 4$.

Conjecture \ref{conj:strongRLF} is known to hold in many cases:
\begin{theorem}[{\cite[Corollary 1.1]{Matsushita13}}]\label{thm:Matsushita}
  Conjecture \ref{conj:strongRLF} holds for all
  irreducible symplectic varieties of K3$^{[n]}$-type and of Kum$^n$-type.
\end{theorem}

\begin{remark}
  Trivially, if $X$ satisfies $\Pic(X) \cap \del\cC_X= 0$, 
   Conjecture \ref{conj:strongRLF} automatically holds for $X$ (since there is no such line bundle $L$).
\end{remark}

We can now formulate the main theorem in its most general form:
\begin{theorem}\label{thm:main-thm}
  Let $X$ be a $2n$-dimensional irreducible symplectic variety  for which ${RR_X}|_{\bZ_{\geq 0}}$ is
  strictly monotonic and such that Conjecture \ref{conj:strongRLF} holds for $X$. Consider a big and nef line
  bundle $H \in \Pic(X)$.
  Then $H$ has a non-trivial base divisor if and only if there exists an irreducible reduced divisor $F$
  of negative square such that $H$ is of the form $H=mL+F$, where $m\geq 2$, $L$ is a
  primitive movable line bundle with $q(L)=0$ and  $(L,F)_q>0$, such that $RR_X(q(H))=\binom{m+n}{n}$.
  In this case $F$ is exactly the fixed divisor of $H$.
\end{theorem}

Let us first note some immediate consequences, before we give the proof in Section \ref{sec:proofofmainthm}.

\begin{corollary}\label{cor:delemptynodivcomp}
  If $X$ is an irreducible symplectic variety  such that
  $\restr{RR_X}{\bZ_{\geq 0}}$ is strictly monotonic and with $\Pic(X)\cap\del\cC_X=0$, then no big and nef line
  bundle on $X$
  has fixed components.  
\hfill $\square$
\end{corollary}

\begin{corollary}\label{cor:2HnobasecompforISV}
  Fix an irreducible symplectic variety $X$  such that  ${RR_X}|_{\bZ_{\geq 0}}$ is
  strictly monotonic 
  and Conjecture \ref{conj:strongRLF} holds for $X$. Let $H \in \Pic(X)$ be a big and nef line
  bundle. Then 
  $2H$ does not have a divisorial base component.
\end{corollary}
\begin{proof}
  For a big and nef line bundle $H\in \Pic(X)$ Theorem \ref{thm:main-thm} shows that if $H$
  has a divisorial base component, $H$ is of the form  $H=mL+F$, where $m\geq 2$, $L$ is movable with
  $q(L)=0$, and $F$ is the fixed part of $H$, which is an irreducible reduced divisor of 
  negative square, and $(L,F)_q>0$.
  
  Therefore $2H=2mL+2F$.  
  Assume for contradiction that $2H$ has a
  divisorial base component $F'$. Then Theorem \ref{thm:main-thm} shows that $2H=m'L'+ F'$,
  where $F'$ is the base locus of $2H$, and $q(L')=0$. However, the base locus of $2H$  is contained in
  the base locus of $H$, which implies $F'=F$. Therefore $m'L'=2H-F=2mL+F$. For the last term use
  $(L,F)_q>0$ and Lemma \ref{lem:ineqforPED} to show that 
  $q(m'L')=q(2mL+F)=4m(L,F)_q+q(F)>0$ which gives the desired contradiction to $q(m'L')=0$.
\end{proof}

%%%%%%%%%%%%%%%%%%%%%%%%%%%%%%%%%%%%%%%%%%%%%%%%%%%%%%%%%%%%%%%%%%%%%%%%%%%%%%%%%%%%%%%%%%
%%%%%%%%%%%%%%%%%%%%%%%%%%%%%%%%%%%%%%%%%%%%%%%%%%%%%%%%%%%%%%%%%%%%%%%%%%%%%%%%%%%%%%%%%%
%%%%%%%%%%%%%%%%%%%%%%%%%%%%%%%%%%%%%%%%%%%%%%%%%%%%%%%%%%%%%%%%%%%%%%%%%%%%%%%%%%%%%%%%%%
\section{Proof of the main theorem (Theorem \ref{thm:main-thm})}\label{sec:proofofmainthm}
In this section we prove the main theorem of this article. The general structure of this proof was inspired by
the case of K3 surfaces (compare \cite[Proof of Corollary 3.15]{HuybrechtsK3}).

\bigskip

A crucial observation in the study of base components for irreducible symplectic varieties is the following:
\begin{lemma}\label{lem:FsquareNeg}
Let $X$ be an irreducible symplectic variety and  $0\neq F  \in \Pic(X)$ be a fixed divisor (i.e. $h^0(X,F)=1$).
  If $RR_X|_{\bZ_{\geq 0}}$ is monotonic and if $X$ satisfies Conjecture \ref{conj:strongRLF} then $q(F)<0$.
\end{lemma}

\begin{proof}
  Suppose for contradiction that $F\in \overline{\cC_X}$. Apply Proposition \ref{prop:RefltoBK} (and its
  proof) to see 
  that, after a finite number of reflections associated to prime exceptional divisors,
  $A\coloneqq R_{D_{k-1}} \circ \dots \circ R_{D_0}(F)\in \birKbar_X$ lies in the closure of the birational
  Kähler cone.
  Recall that in the proof of Proposition \ref{prop:RefltoBK},  we set $F_0\coloneqq F$. As long as
  $F_i\notin \birKbar_X$, we
  successively pick effective divisors $D_i$ with $q(D_i)<0$ and $(D_i,F_i)_q<0$, and define
  $F_{i+1}\coloneqq R_{D_i}(F_i) = F_i - \frac{2(D_i,F_i)_q}{q(D_i)}  \, D_i$.
  By the choice of the $D_i$, each of these reflections
  subtracts a positive multiple of the effective divisor 
   $D_i$. 

   Consequently
  \begin{equation*}
    F= A+\sum_{i=0}^{k-1} a_i D_i \quad {\rm with\ } a_i= \frac{2(D_i,F_i)_q}{q(D_i)} >0.
  \end{equation*}
  Note that $a_i\in \bZ$ is integral by Lemma \ref{lem:ineqforPED}.
Therefore, it suffices to show that no element $A\in \birKbar_X$ is a fixed divisor.

If $q(A)>0$, then by assumption $RR_X(q(A))\geq RR_X(0)=n+1>1$, which together with Lemma \ref{lem:h0inbirKbar} implies 
\begin{equation*}
  h^0(X,A)=\chi(X,A)=RR_X(q(A))>1.
\end{equation*}
In particular $A$ is not a fixed divisor.

If $q(A)=0$ Conjecture \ref{conj:strongRLF} implies that $A$ is movable. 
In both cases this shows that $F$ could not be a fixed divisor.
\end{proof}

Using Lemma \ref{lem:FsquareNeg} we can show:

\begin{lemma}\label{lem:Msquare0}
  Let $X$ be an irreducible symplectic variety which satisfies Conjecture \ref{conj:strongRLF} and such
  that the map ${RR_X}|_{\bZ_{\geq 0}}$ is strictly
  monotonic. Fix a big and nef  bundle $H\in \Pic(X)$ and consider the
  decomposition $H=M+F$ into the movable and fixed part. 
  If in this situation $q(M)>0$, then $F=0$.
\end{lemma}
\begin{proof}
  By assumption, $M$ is movable.
  Therefore $M \in \Pic(X)\cap\birKbar_X$ by Corollary \ref{cor:MovbirKbar}. Since $q(M)>0$, Lemma
  \ref{lem:h0inbirKbar} applies and yields
\begin{align*}
  \chi(X,H)=h^0(X,H)=h^0(X,M)=\chi(X,M),
\end{align*}
where the first equality exploits that $H$ is big and nef, and therefore satisfies 
Kodaira vanishing,
the second equality holds since $M$ is the movable part of $H$, and the third  equality is the statement
of Lemma \ref{lem:h0inbirKbar}.

This implies $q(H)=q(M)$, since by assumption $RR_X|_{\bZ\geq 0}$ is strictly monotonic, and in
particular injective. 
Then 
\begin{align*}
  q(M)&=q(H)=q(M+F)=q(M)+2(M,F)_q+q(F) \quad{\rm implies}\\
  0&=2(M,F)_q+q(F)=\underbrace{(M,F)_q}_{\geq 0}+\underbrace{(H,F)_q}_{\geq 0}.
\end{align*}
The fact that both terms on the right hand side are at least zero, follows from Lemma \ref{lem:BBforeffective}. 
This equality implies that $(M,F)_q = 0 = (H,F)_q$. 
Therefore
\begin{equation*}
  0= (H,F)_q =(M+F,F)_q = \underbrace{(M,F)_q}_{=0} + \, q(F).
\end{equation*}
By Lemma \ref{lem:FsquareNeg}, we know that $q(F)<0$ whenever $F$ is not trivial. It follows that $F=0$
as claimed.  
\end{proof}

Using these results, we can complete the proof of the main theorem
\begin{proof}[{Proof of Theorem \ref{thm:main-thm}}]
  Let $X$ be an irreducible symplectic variety satisfying the conditions of Theorem \ref{thm:main-thm}
  and $H=M+F\in \Pic(X)$ be a big and nef line bundle with its decomposition into movable and fixed part.
  Suppose that  $H\neq M$ or equivalently that the fixed divisor $F\neq 0$ is not trivial. It follows from Lemma
  \ref{lem:Msquare0} that $q(M)=0$.
  Using that $RR_X|_{\bZ_{\geq 0}}$ is strictly monotonic one obtains
  \begin{equation*}
    h^0(X,M)=h^0(X,H)=\chi(X,H)=RR_X(q(H))>RR_X(0)=n+1,
  \end{equation*}
  since $M$ is the movable part of $H$, by Kodaira vanishing, Theorem \ref{thm:RRISV}, and Proposition
  \ref{prop:chi0}. 
  Denote $M=m\cdot L$, where $L$ is the
  primitive line bundle in the same ray and $m\geq 0$. By Conjecture \ref{conj:strongRLF}, $|L|$ induces a rational
  Lagrangian fibration  
  to $\bP^n$ and therefore $h^0(X,L)=n+1$. Thus $m\geq 2$.

We need to observe that $F$ is irreducible. 
 Let $F=\sum a_i F_i$ be the decomposition of $F$ into its
irreducible components. 
Lemma \ref{lem:FsquareNeg} shows that $q(F)<0$.
Since $0<(H,F)_q=(M,F)_q+ q(F)= \sum_i a_i(M,F_i)_q+ q(F) $,  there exists at least one $i_0$ such that the irreducible component
$F_{i_0}$ satisfies $(M,F_{i_0})_q>0$. 

Define $H'\coloneqq M+F_{i_0}=mL+F_{i_0}$. We will show that $H'\in \birKbar_X\cap \cC_X$. 
By choice of $i_0$ one has $(L,F_{i_0})_q>0$, and thus $(L,F_{i_0})_q\geq \div(F_{i_0})>0$. 
On the other hand, since $F_{i_0}$ is a fixed divisor, Lemma \ref{lem:FsquareNeg} shows that $q(F_{i_0})<0$, and thus Lemma
\ref{lem:ineqforPED}  implies that 
\begin{equation}\label{eq:ineqforFi0}
-q(F_{i_0})\leq 2\div(F_{i_0})\leq 2(L,F_{i_0})_q\,.
\end{equation}

Consequently
\begin{align*}
  q(H')=q(mL+F_{i_0})
  =\underbrace{m^2q(L)}_{=0}+  2(m-1)(L,F_{i_0})_q+ \underbrace{2(L,F_{i_0})_q+ q(F_{i_0})}_{\geq 0}>0.
\end{align*}
For any irreducible effective divisor $D\neq F_{i_0}\in \Pic(X)$ (in particular for uniruled divisors) 
use Lemma \ref{lem:BBforeffective} to see
\begin{align*}
  &(H',D)_q=(mL+F_{i_0},D)_q=\underbrace{m(L,D)_q}_{\geq 0}+\underbrace{(F_{i_0},D)_q}_{\geq 0} \geq 0.
\end{align*}
Finally observe that
\begin{align*}
  (H',F_{i_0})_q=(mL+F_{i_0},F_{i_0})_q=m(L,F_{i_0})_q+q(F_{i_0}) \geq 0,
\end{align*}
where the last inequality follows again from \eqref{eq:ineqforFi0}.
Thus $H'\in \birKbar_X\cap \cC_X$ by Theorem \ref{thm:descrofbirKbar}.

Similar to Lemma \ref{lem:Msquare0} we show that $F-F_{i_0}=0$: Since $M$ is the movable part of $H$, 
observe that $h^0(X,H)=h^0(X,H')=h^0(X,M)$.  
Using Lemma \ref{lem:h0inbirKbar} twice, thus gives $\chi(X,H)=h^0(X,H)=h^0(X,H')=\chi(X,H')$, and thus
$q(H)=q(H')$ by the injectivity of $RR_X|_{\bZ_{\geq 0}}$.
Therefore
\begin{align*}
  0&=q(H)-q(H')
  = (H, H'+F-F_{i_0})_q - q(H')\\
  &= (H'+F-F_{i_0}, H')_q +(H, F-F_{i_0})_q  -q(H')\\
  &=\underbrace{(F-F_{i_0},H')_q}_{\geq 0} + 
    \underbrace{(H,F-F_{i_0})_q}_{\geq 0},
\end{align*}
where the last inequalities follow from  Corollary \ref{cor:BBforbig+nef}.
Consequently $(H,F-F_{i_0})_q = 0 = (H',F-F_{i_0})_q$, and thus
\begin{equation*}
q(F-F_{i_0})= (H-H',F- F_{i_0})_q = (H,F- F_{i_0})_q-(H',F- F_{i_0})_q = 0.
\end{equation*}
Since $F-F_{i_0}$ is a fixed divisor, Lemma \ref{lem:FsquareNeg} implies that $F-F_{i_0}=0$. Therefore 
 $F=F_{i_0}$ is irreducible and reduced and in addition $(F,L)_q=(F_{i_0},L)_q>0$.

In order to establish the condition on $RR_X(q(H))$, 
  begin with the following observation.
  For a primitive movable line bundle $L\in \Pic(X)\cap \birKbar_X$ with $q(L)=0$, use Conjecture
  \ref{conj:strongRLF} for $X$ to obtain a birational model $X'$, where the corresponding line bundle $L'$ 
  induces a Lagrangian fibration to $\bP^n$.
  In particular 
  \begin{equation}\label{eq:h0ofmL}
   h^0(X,mL)= h^0(X',mL')=h^0(\bP^n, \dO(m))= \binom{m+n}{n},
  \end{equation}
  where the first equality follows from the fact that $X$ and $X'$ are isomorphic away from codimension
  two by Lemma \ref{lem:isocodim1}. 

Therefore the theorem follows from the above, since by Kodaira vanishing
$h^0(X,H)=\chi(X,H)=RR_X(q(H))$, and because a line bundle of the form $H=mL+F$ has base
locus along $F$ if and only if $h^0(X,H)=h^0(X,mL)$.
  \end{proof}

%%%%%%%%%%%%%%%%%%%%%%%%%%%%%%%%%%%%%%%%%%%%%%%%%%%%%%%%%%%%%%%%%%%%%%%%%%%%%%%%%%%%%%%%%%%%%%%%%%%%
%%%%%%%%%%%%%%%%%%%%%%%%%%%%%%%%%%%%%%%%%%%%%%%%%%%%%%%%%%%%%%%%%%%%%%%%%%%%%%%%%%%%%%%%%%%%%%%%%%%%
\section{Discussion of the assumptions in the statement}\label{sec:discuss-conditions}
In this section we discuss the assumptions in the statement of Theorem \ref{thm:main-thm}.

First note that Theorem \ref{thm:main-thm} applies to all varieties of K3$^{[n]}$-type and of
Kum$^n$-type: Conjecture \ref{conj:strongRLF} holds in these cases by Theorem \ref{thm:Matsushita}.
The
condition that $RR_X|_{\bZ_{\geq 0}}$ is strictly monotonic, can either be verified directly, or deduced
from the following lemma:

Let $U$ be the standard hyperbolic lattice of rank two, i.e.~the lattice with associated matrix 
$(\begin{smallmatrix} 0&1\\ 1& 0 \end{smallmatrix})$. 

\begin{lemma}\label{lem:RRinj}
  Fix an irreducible symplectic variety $X$ such that 
  \begin{compactenum}
    \item Conjecture \ref{conj:strongRLF} is satisfied for all deformations of $X$,
    \item $H^2(X,\bZ)$ is an even lattice with respect to the Beauville--Bogomolov--Fujiki form,
      and \label{it:even} 
    \item $H^2(X,\bZ)$ contains a copy of $U$.\label{it:U}
  \end{compactenum}
  Then the restriction 
  $\restr{RR_X}{\bZ_{\geq 0}}\colon\bZ_{\geq 0}\to \bZ$ is strictly monotonic and in particular injective. 
\end{lemma}

\begin{remark}
  This lemma should serve as an indication that it is reasonable to ask whether $\restr{RR_X}{\bZ_{\geq
      0}}$ is strictly monotonic for all irreducible symplectic varieties.
  Conjecture \ref{conj:strongRLF} is a classical conjecture on irreducible symplectic
  varieties. Furthermore all known examples of irreducible symplectic varieties satisfy \ref{it:even} and 
  \ref{it:U}. It is not known, whether these conditions hold for arbitrary irreducible
  symplectic varieties. 
  I expect, that in general the question whether  $\restr{RR_X}{\bZ_{\geq
      0}}$ is strictly monotonic should be easier to answer than Conjecture \ref{conj:strongRLF}.
\end{remark}

\begin{proof}
The polynomial $RR_X$ only depends on the deformation type of $X$. 
Therefore by passing to a suitable deformation of $X$ we may assume that 
$\Pic(X)\iso U$. Let $E,F\in \Pic(X)$ be the generators of $U$
that span the boundary of the positive cone. 
For $k>0$ define the element $B_k\coloneqq F+kE\in U$. Then $q(B_k)=2k$ and therefore 
\begin{equation*}
RR_X(2k)=RR_X(q(B_k))=\chi(X,B_K).
\end{equation*}

The closure of the birational Kähler cone $\birKbar_X$ is cut out by uniruled irreducible (and reduced)
divisors $D$ with $q(D)<0$ 
(compare Theorem \ref{thm:descrofbirKbar}). By Lemma \ref{lem:ineqforPED} every such divisor $D$
satisfies $\half \left|q(D)\right| \leq \div(D)$. 
Express $D$ as $aE+bF$ for some $a,b\in \bZ$. Then the equality gives
\begin{equation*}
  \half \left|2ab\right| \leq \gcd(a,b), 
\end{equation*}
which has solutions for $\left| a\right|=\left| b\right|=1$ or $a=0$ or $b=0$. Thus the only solution
with $q(D)<0$ is $\pm (E-F)$.

Therefore there are two  possible cases:

{\it Case 1: Both elements $E-F$ and $F-E$ are not effective, and thus $\birKbar_X=\cC_X$.}
By Conjecture \ref{conj:strongRLF} both $E$ and $F$ are effective with
$h^0(X,E)=h^0(X,F)=n+1>1$, where $n=\dim(X)/2$ (both are giving a rational Lagrangian fibration to $\bP^n$).
In particular observe that
\begin{align*}
  &h^0(X,B_1)= h^0(X,E+F)>h^0(X,F)=n+1,\quad{\rm and} \\
  &h^0(X,B_{k+1})=h^0(X,B_k+E)>h^0(X,B_k)\quad \forall k>0.
\end{align*}

%The precise argumentation here is: Pick a section $s\in h^0(X,B_k)$ and $x,y\in B$ such that their fibres
%are not contained in the zero locus of $s$. Then $\otimes s_x\colon h^0(X,B_k)\inj h^0(X,B_k+E)$ and $s\otimes
%s_y$ is not contained in the image.

{\it Case 2: $\birKbar_X\neq \cC_X$.}
In this case there exists a unique prime exceptional divisor: either $E-F$ or $F-E$.
 Without loss of generality we may assume $D\coloneqq F-E$ is effective, and thus  $E \in \birKbar_X$ induces a rational
Lagrangian fibration to
$\bP^n$ by Conjecture \ref{conj:strongRLF}. Therefore $h^0(X,E)=n+1>1$, and $F=E+D$ implies 
$h^0(X,F)\geq h^0(X,E)=n+1>1$.
Thus
\begin{align*}
  &h^0(X,B_1)=h^0(X,F+E)=h^0(X,2E+D)\geq h^0(X,2E) > n+1,\quad{\rm and}\\
  &h^0(X,B_{k+1})=h^0(X,B_k+E)>h^0(X,B_k)\quad \forall k>0.
\end{align*}

In both cases the elements $B_k$ are contained in $\birKbar_X\cap \cC_X$ (for Case 2 note that
$(B_k,D)_q=(F+kE,F-E)_q=k-1\geq 0$). 
Therefore by Lemma \ref{lem:h0inbirKbar}
  $\chi(X,B_k)=h^0(X,B_k)$.

Together, this  yields that in both cases:
\begin{align*}
  &RR_X(2)=\chi(X,B_1)=h^0(X,B_1)>n+1=RR_X(0), \quad {\rm and}\\
  &RR_X(2k+2)=\chi(X,B_{k+1})=h^0(X,B_{k+1})>\\
  &\hspace{6em} h^0(X,B_k)=\chi(X,B_k)=RR_X(2k) \quad \forall k>0,
\end{align*}
which concludes the proof.
\end{proof}

For arbitrary deformation types, the condition that $RR_X$ is strictly monotonic is known to hold up to
dimension 6:
\begin{proposition}
  Let $X$ be an irreducible symplectic variety of dimension $\leq 6$.
  Then in the Riemann--Roch polynomial
  $RR_X(x)\coloneqq \sum_{i=0}^n b_i x^i$
  all $b_i$ are non-negative.
  Since $b_n>0$ this implies that $RR_X$ is strictly monotonic.
\end{proposition}
\begin{proof}
  For dimension $6$ this is shown in \cite{Cao-Jiang}. In fact the proof of \cite[Corollary 3.4]{Cao-Jiang}
  gives an explicit 
  description for the $b_i$ in this case and shows that they are non-negative.

  The same arguments work for dimension $<6$.

  Finally, let us explain, why $b_n>0$.
  From the proof of \cite[Corollary 23.18]{Gross-Huybrechts-Joyce} one sees that
  $\int_X \alpha^{2n}=(2i)!b_nq(\alpha)^i$ for every $\alpha \in H^2(X,\bZ)$.
  Therefore $b_n$ differs by a positive factor from the Fujiki constant, which is known to be positive.
\end{proof}

%%%%%%%%%%%%%%%%%%%%%%%%%%%%%%%%%%%%%%%%%%%%%%%%%%%%%%%%%%%%%%%%%%%%%%%%%%%%%%%%%%%%%%%%%%%%%%%%%%
%%%%%%%%%%%%%%%%%%%%%%%%%%%%%%%%%%%%%%%%%%%%%%%%%%%%%%%%%%%%%%%%%%%%%%%%%%%%%%%%%%%%%%%%%%%%%%%%%%
\section{Behaviour under deformation}\label{sec:defo}
In this section we discuss the behaviour of divisorial base loci of big and nef line bundle on irreducible symplectic
varieties under deformations.

For a given irreducible symplectic variety $X$, fix a marking
$\eta\colon H^2(X,\bZ) \overset{\iso}{\too} \Lambda$ for some lattice $\Lambda$.
Denote by $\cM_\Lambda^X$ the moduli space of marked irreducible symplectic varieties deformation equivalent
to $(X,\eta)$.
Pick an element $h\in\Lambda$, such that $H\coloneqq \eta^{-1}(h)$ is a big and nef line bundle on $X$.
Denote by $\cM_{\Lambda,h}^X$ the moduli space of semi-polarized marked irreducible symplectic varieties
deformation equivalent to $(X,\eta)$, i.e.\,the subspace of the moduli space $\cM_\Lambda^X$, where the class
$h$ still corresponds to a  big and nef line bundle. % Normalerweise semi-polarized: H is nef. In this case
% big is automatic.
To fix notation, $t\in \cM_{\Lambda,h}^X$ parametrizes an irreducible symplectic variety $X_t$ with marking
$\eta_t\colon H^2(X_t,\bZ)\iso\Lambda$ such that $H_t\coloneqq \eta_t^{-1}(h)$ is a big and nef line bundle.

\begin{theorem}\label{thm:defoofBD}
  Let $X$ be an irreducible symplectic variety such that its deformation type satisfies the conditions of the
  main theorem (Theorem \ref{thm:main-thm}).
  Fix a marking $\eta$ and a class $h\in \Lambda$ as above.
  Then the locus in the semi-polarized moduli space $\cM_{\Lambda,h}^X$, where the semi-polarization has a non-trivial
  base divisor, is a union of disjoint Noether--Lefschetz divisors (the Noether--Lefschetz divisors corresponding to the
  possible base divisors).
\end{theorem}

\begin{proof}
  First note that the condition that $RR_{X_t}$ is strictly monotonic ensures that for each $t\in
  \cM_{\Lambda, h}^X$ the
  semi-polarization $H_t$ is effective (since $H_t$ is big and nef, it satisfies Kodaira vanishing and thus
   $h^0(X_t, H_t)=\chi(X_t, H_t)=RR_{X_t}(q(H_t))>0$ by assumption). Thus it makes sense to study
  the base divisor.

  Consider a small open subset $T\subseteq \cM_{\Lambda,h}^X$ such that there exists a universal family
  $\pi\colon \cX \to T$ with semi-polarization $\cH\in \Pic(\cX)$ (this is possible by
  \cite[1.12]{Huybrechts:HK:basic-results}).
  Let $X_t\coloneqq \pi^{-1}t$ and $H_t\coloneqq \cH|_{X_t}$.
  Then again by Kodaira vanishing $H^k(X_t, H_t)=0$ for all $t\in T$ and $k>0$.
  Therefore, by base change  $\pi_*(\cH)$ 
  is a vector bundle with $H^0(X_t,H_t)=\pi_*\cH|_t$.  Thus after shrinking $T$, one can assume that all
  global sections of $H_t$ are restrictions of global 
  sections of $\cH$.
  This implies that the base locus of $\cH$, which is a closed
  set, is equal to the union of the base loci of the
  $H_t$.
  In particular, the locus in $T$ where  $H_t$ has a base divisor is closed.
  It thus suffices to prove the statement locally in $\cM_{\Lambda,h}^X$, and we can assume that 
  there exists a universal family.

  If there exists $t_0$ such that $H_{t_0}$ has non-trivial divisorial base locus, then
  one can apply the main theorem (Theorem \ref{thm:main-thm}). One sees that $H_{t_0}$ is of the form
  $H_{t_0}=mL_{t_0} + F_{t_0}$, where $m\geq 2$, $L_{t_0}$ is a
  primitive movable line bundle with $q(L_{t_0})=0$ and  $(L_{t_0},F_{t_0})_q>0$, such that
  $h^0(X_{t_0}, H_{t_0})=RR_{X_{t_0}}(q(H_{t_0}))=\binom{m+n}{n}$, $F_{t_0}$ is a prime exceptional divisor with
  $q(F_{t_0})<0$.

  Consider $T'\subset T$ the Noether--Lefschetz locus, where the
  class of $F_{t_0}$ stays $(1,1)$, and thus there is a deformation $\cF\in \Pic(\cX_{T'})$. We will show,
  that $H_t$ has non-trivial base divisor for all $t\in T'$, and that the base divisor is exactly
  $F_t\coloneqq \cF|_{X_t}$.

  As observed above, Kodaira vanishing implies that $h^0(X_t, H_t)$ is constant and thus
  \begin{equation}\label{eq:h0isconstant}
    h^0(X_t, H_t)=h^0(X_{t_0}, H_{t_0})=\binom{m+n}{n}.
  \end{equation}

  Since $F_{t_0}$ is a prime exceptional divisor, it stays effective under deformation by \cite[Proposition 5.2]{MarkmanPED},
  i.e.~for every $t\in T'$ the divisor $F_t$ is effective.

Set $L_t\coloneqq \frac{1}{m}(H_t-F_t)$, which is a primitive, integral class. Since both $H_t$ and $F_t$ are
of $(1,1)$ type along $T'$, the same holds for $L_t$.
\begin{claim}\label{claim:h0L-t}
  For every $t\in T'$ the line bundle  $L_t$ is effective with
  $h^0(X_t, mL_t)\geq h^0(\bP^n, \dO_{\bP^n}(m))=\binom {m+n}{n}$ for all $m$. 
\end{claim}
\begin{proofofclaim}
  Since $L_{t_0}\in {\birKbar}_{X_{t_0}}$, also for very general $t\in T'$ the line bundle
   $L_t\in {\birKbar}_{X_t}$
  (use Theorem \ref{thm:descrofbirKbar} to see this is satisfied for every
  $t$ which does not admit an additional uniruled divisor in $\Pic(X_t)$).   Further, note that $q(L_t)=q(L_{t_0})=0$.
  Consequently, for general $t\in T'$ Conjecture \ref{conj:strongRLF} implies that $L_t$ induces a rational Lagrangian 
  fibration to $\bP^n$, and thus $h^0(X_t, mL_t)\geq h^0(\bP^n, \dO_{\bP^n}(m))=\binom {m+n}{n}$.
  The claim follows by semicontinuity.
\end{proofofclaim}

To conclude the proof of Theorem \ref{thm:defoofBD} use that by construction $H_t=mL_t +F_t$ for every $t\in T'$,
where we saw above that all involved line bundles are effective. We therefore have
\begin{equation*}
  \binom{m+n}{n}\overset{\eqref{eq:h0isconstant}}{=}
  h^0(X_t, H_t)
  \geq h^0(X_t, mL_t)
  \overset{\ref{claim:h0L-t}}{\geq}\binom{m+n}{n}.
\end{equation*}
Thus all involved inequalities must be equalities, in particular $h^0(X_t,H_t)=h^0(X_t,mL_t)$, which
implies that $F_t$ is contained in the base locus of $H_t$ and then Theorem \ref{thm:main-thm} implies
that the base divisor of $H_t$ is irreducible and reduced, and therefore coincides with $F_t$.
\end{proof}

\TODO{be really sure, that NL-loci usually only refer to staying (1,1).}

    \TODO{rauskriegen, ob es eine standardisierte Schreibweise für Noether--Lefschetz gibt - statt $T_F$.}

%%%%%%%%%%%%%%%%%%%%%%%%%%%%%%%%%%%%%%%%%%%%%%%%%%%%%%%%%%%%%%%%%%%%%%%%%%%%%%%%%%%%%%%%%%%%%%%%%%%%
%%%%%%%%%%%%%%%%%%%%%%%%%%%%%%%%%%%%%%%%%%%%%%%%%%%%%%%%%%%%%%%%%%%%%%%%%%%%%%%%%%%%%%%%%%%%%%%%%%%%
\section{Base divisors for K3$^{[n]}$-type} %\label{sec:basedivforspecial}
\label{sec:basedivforK3n}

In this section we improve the main theorem (Theorem \ref{thm:main-thm}) for irreducible
symplectic varieties of  K3$^{[n]}$-type. We need the following Riemann--Roch formula.

\begin{proposition}[{Riemann--Roch for K3$^{[n]}$-type, \cite[Example 23.19]{Gross-Huybrechts-Joyce}}]\label{prop:EGL}
    Let $X$ be of K3$^{[n]}$-type, and $L \in \Pic(X)$ be a line bundle. Then
  \begin{equation*}
    \chi(X,L)=\bigg(\twolines{\frac{1}{2}q(L)+n+1}{n}\bigg).
  \end{equation*}
\end{proposition}

With this in mind, we find the following characterization for base divisors of big and nef line bundles on
irreducible symplectic varieties of K3$^{[n]}$-type.

\begin{proposition}\label{prop:basecomponentsK3n-type}
  Let $X$ be an irreducible symplectic variety of K3$^{[n]}$-type and $H\in \Pic(X)$ a line bundle that
  is big and nef.  
  Then  $H$ has a fixed divisor if and only if $H$ is of the form
  $H=mL+F$, where $m\geq 2$, $L$ is movable with $q(L)=0$, and $F$ is an irreducible reduced divisor of
  negative square with $(L,F)_q=1$. In this case $F$ is the fixed divisor of $H$.
\end{proposition}

\begin{proof}
  Let $H\in \Pic(X)$ be a big and nef line bundle on $X$. Recall that Theorem \ref{thm:main-thm} can be
  applied for K3$^{[n]}$-type by Section \ref{sec:discuss-conditions}. This shows that $H$
  has non-trivial base divisor if and only if $H=mL+F$ where $m\geq 2$, $L$ is movable with $q(L)=0$, and $F$
  is an irreducible reduced divisor of negative square, and $(L,F)_q>0$, such that $h^0(X, H)=\binom {m+n}{n}$. 
  Therefore, we only need to observe that the additional condition $(L,F)_q=1$ is equivalent to
  $h^0(X, H)=\binom {m+n}{n}$ .

  Since $H$ is big and nef by assumption, Kodaira vanishing and the Riemann--Roch for
  K3$^{[n]}$-type (see Proposition \ref{prop:EGL}) imply that 
  \begin{equation*}
    h^0(X,H)= \chi(X,H)=\binom{\frac{1}{2}q(H)+n+1}{n}.
  \end{equation*}
  In this case $H=mL+F$ shows that 
  \begin{equation}\label{eq:qofH}
    q(H)=q(mL + F) = 2m  (L,F)_q + q(F)=  2(m-1) (L,F)_q + \underbrace{2(L,F)_q+ q(F)}_{\geq 0} ,
  \end{equation}
where the inequality of the last term follows from Lemma \ref{lem:ineqforPED} since $(L,F)_q>0$.

Therefore 
\begin{align}
   h^0(X,H)&=\binom{\frac{1}{2}q(H)+n+1}{n}
 % = \binom{(m-1) (L,F)_q + (L,F)_q+ \half q(L)+n+1}{n}
\geq \binom{(m-1) (L,F)_q +n+1}{n} \notag \\
&\overset{(*)}\geq \binom{(m-1) +n+1}{n}
= \binom {m+n}{n}, \label{eq:inequforh0H}
\end{align}
where equality in  $(*)$,  holds if only if $(L,F)_q=1$.

\bigskip

  To see the equivalence $h^0(X,H)=\binom{m+n}{n}$ if and only if $(L,F)_q=1$, first suppose 
  $ h^0(X,H)= \binom{m+n}{n}$.
  In this case equality holds in all places in \eqref{eq:inequforh0H}.
In particular, there is equality in $(*)$, whence $(L,F)_q=1$.

For the other implication, suppose that $(L,F)_q=1$.
Since $q(F)<0$,  Lemma \ref{lem:BBforeffective} implies that $F$ is a fixed divisor,
 $(L,F)_q=1$ ensures that $\div(F)=1$.
Note that the K3$^{[n]}$-lattice $\LambdaKEn$ is an even lattice. Therefore use  Lemma \ref{lem:ineqforPED}
to see that $-q(F)\leq 2\div(F)=2$, and thus $q(F)=-2$.
In particular, $2(L,F)_q+q(L)=2\cdot 1 -2=0$, which shows that both inequalities in
\eqref{eq:inequforh0H} are equalities in this case.
Consequently, $h^0(X,H)= \binom{m+n}{n}$.
\end{proof}

%%%%%%%%%%%%%%%%%%%%%%%%%%%%%%%%%%%%%%%%%%%%%%%%%%%%%%%%%%%%%%%%%%%%%%%%%%%%%%%%%%%%%%%%%%%%%%%%%%
%%%%%%%%%%%%%%%%%%%%%%%%%%%%%%%%%%%%%%%%%%%%%%%%%%%%%%%%%%%%%%%%%%%%%%%%%%%%%%%%%%%%%%%%%%%%%%%%%%
\section{Base divisors for Kum$^n$-type}\label{sec:basedivforKumn}

Similar as in the previous section, we need the following
\begin{proposition}[{Riemann--Roch for Kum$^n$-type, \cite[Example 23.20]{Gross-Huybrechts-Joyce}}]\label{prop:RRKum}
    Let $X$ be an irreducible symplectic variety of Kum$^n$-type, and $L \in \Pic(X)$ be a line bundle. Then
  \begin{equation*}
    \chi(X,L)=(n+1) \bigg(\twolines{\frac{1}{2}q(L)+n}{n}\bigg).
  \end{equation*}
\end{proposition}

This can be used to exclude divisorial base components for big and nef line bundles on varieties of Kum$^n$-type:

\begin{proposition}\label{prop:basecomponentsKumn-type}
  Let $X$ be an irreducible symplectic variety of Kum$^n$-type, and $H$ a big and nef line bundle. Then $H$ is movable, i.e.~it does not have a fixed divisor.
\end{proposition}

\begin{proof}
  Suppose for contradiction that $H$ has non-trivial base divisor. Note that Theorem \ref{thm:main-thm} can be
  applied to Kum$^n$-type by Section \ref{sec:discuss-conditions}.
  Thus  there exists an
  irreducible reduced divisor $F$ 
  of negative square such that $H$ is of the form $H=mL+F$, where $m\geq 2$, $L$ is a
  primitive movable line bundle with $q(L)=0$ and  $(L,F)_q>0$, such that $RR_X(q(H))=\binom{m+n}{n}$.
  We only need to see that this cannot happen for $X$ of Kum$^n$-type.

  As in \eqref{eq:qofH}
    \begin{equation*}
    q(H) = 2(m-1) (L,F)_q + \underbrace{2(L,F)_q+ q(F)}_{\geq 0}.
    \end{equation*}
    We will use this to show, that  there are no positive integral values for $m$ and $n$ such that
    \begin{equation}\label{eq:contr}
      RR_X(q(H))=\binom{m+n}{n}=\frac{(m+n)!}{n!\cdot m!}.
    \end{equation}
    Distinguish between two cases:

\bigskip

    {\it Case 1: If $(L,F)_q=1$:}
    Since the Kum$^n$-lattice is even and $q(F)<0$, in this case
    $0\leq 2(L,F)_q+ q(F)=2 + q(F)\leq 0$, and thus in fact there is equality everywhere, $q(F)=-2$ and
    \begin{equation}\label{eq:qofHcase1}
      q(H)=2(m-1)(L,F)_q=2(m-1).
    \end{equation}
    With this in mind use Riemann--Roch for Kum$^n$-type (Proposition
    \ref{prop:RRKum}) to see that the left hand side of \eqref{eq:contr} is
    \begin{align*}
      RR_X(q(H))&=\chi(X,H)=(n+1) \binom{\frac{1}{2}q(H)+n}{n}
      \overset{\eqref{eq:qofHcase1}}{=} (n+1) \binom{m-1+n}{n}\\
      &=(n+1)\frac{(m+n-1)!}{n!\cdot(m-1)!}
    \end{align*}
    Therefore, \eqref{eq:contr} becomes
    \begin{align*}
      &(n+1)\frac{(m+n-1)!}{n!\cdot(m-1)!} = \frac{(m+n)!}{n!\cdot m!}\\
      \iff \hspace{2em} & (n+1)\cdot 1 = \frac{(m+n)}{ m}\\
      \iff \hspace{2em} & nm + m = m+n\\
      \iff \hspace{2em} & 1=m .
    \end{align*}
    which is a contradiction to $m\geq 2$, which would hold if $H$ were a big and nef line bundle with a fixed
    divisor. 
     
    \bigskip

    {\it Case 2: If $(L,F)_q\geq 2$:}
    Again, we want to verify, that there are no positive integral solutions $m,n$ of \eqref{eq:contr}.
    Note that in this case
    \begin{align*}
      q(H)= 2(m-1) (L,F)_q + \underbrace{2(L,F)_q+ q(F)}_{\geq 0}
      \geq 4(m-1).
    \end{align*}
    In addition,  $m\geq 2$ immediately implies that $2(m-1)\geq m$.
    Since $RR_X$ is monotonic for Kum$^n$-type this implies that
    \begin{align*}
      RR_X(q(H))&=n\binom{\frac{1}{2}q(H)+n-1}{n-1}
      \geq (n+1)\binom{2(m-1)+n}{n}
      \geq(n+1)\binom{m+n}{n}\\
      &> \binom{m+n}{n}.
    \end{align*}
    This is the desired contradiction to the equality in \eqref{eq:contr}.
\end{proof}
   \TODO{Case 1 zeigt wahrscheinlich, dass es gar prime exceptional divisors mit $\div=1$ geben
     kann... Überlegen, ob ich das ausarbeite und als Kommentar hier einfüge.}

%%%%%%%%%%%%%%%%%%%%%%%%%%%%%%%%%%%%% References %%%%%%%%%%%%%%%%%%%%%%%%%%%%%%%%%%%%%
%%%%%%%%%%%%%%%%%%%%%%%%%%%%%%%%%%%%%%%%%%%%%%%%%%%%%%%%%%%%%%%%%%%%%%%%%%%%%%%%
\bibliographystyle{alpha}
\bibliography{\folder Literatur}

\end{document}